\documentclass[12pt]{article}
\usepackage{amssymb,latexsym, amsmath, enumerate, amsthm, mathscinet, mathtools}
\usepackage{color}
\usepackage{hyperref}
\hypersetup{
    CJKbookmarks=true,
	colorlinks=true,
	linkcolor=blue,
	filecolor=blue,      
	urlcolor=blue,
	citecolor=blue,
}

\def\R{\mathbb R}
\def\N{\mathbb N}

\def\G{\mathbb G}
\def\H{\mathbb H}
\def\X{\mathrm X}

\def\S{\mathcal S}
\def\SS{\mathbb S}
\def\M{\mathcal M}
\def\W{\mathcal W}
\def\g{\mathfrak{g}}
\def\L{\mathbb L}
\def\J{\mathbb J}
\def\P{\mathbf P}
\def\QQ{\mathbf Q}
\def\bL{\mathbf L}
\def\q{\mathbf q}
\def\BL{\mathrm {BL}}

\def\CC{\mathbf C}

\def\DP{{\mathcal D}^{+,\mathcal{P}}}

\def\RR{\mathbf R}

\def\h{\hat}

\def\supp{\mathrm{supp}}
\def\dim{\mathrm{dim}}

\def\Var{\mathrm{Var}}
\def\PP{\mathrm{P}}
\def\BV{\mathrm{BV}}
\def\F{\mathcal{F}}

\def\id{\mathrm{id}}
\def\U{\mathbb U}
\def\V{\mathbb V}

\newtheorem{theorem}{Theorem}
\newtheorem{lemma}{Lemma}
\newtheorem{remark}{Remark}

\newtheorem{proposition}{Proposition}

\numberwithin{equation}{section}

\title{Loomis--Whitney inequalities on Reiter--Heisenberg groups}
\author{Sheng-Chen Mao, Ye Zhang}

\date{}

\begin{document}

\renewcommand{\theequation}{\thesection.\arabic{equation}}
 \maketitle

\vspace{-0.5cm}

\bigskip
	
{\bf Abstract.} We establish a Loomis--Whitney inequality for the Reiter--Heisenberg groups $\G_{qp}$, a family of step-two Carnot groups that includes the Heisenberg groups when $q=1$. The proof is based on the duality between Brascamp--Lieb inequalities and entropy subadditivity: we first derive the result for $\G_{q1}$ from the known inequality on the first Heisenberg group, using conditional entropy and the invariance of differential entropy under volume-preserving diffeomorphisms; then we pass from $\G_{q1}$ to $\G_{qp}$ via a stability principle for Loomis--Whitney inequalities under finite central sums, which generalizes the argument in (Zhang, 2024~\cite{Z24}).  As consequences, we obtain the associated geometric projection inequality, a Gagliardo--Nirenberg--Sobolev inequality, and an isoperimetric inequality.
	
	\medskip
	
	
	\medskip
	

\bigskip

\section{Introduction}

\subsection{Background and motivation}\label{ss11}

The classical Loomis--Whitney inequality in $\R^k (k \ge 2)$ states that
\begin{align}\label{LWc}
|E| \le \prod_{j = 1}^k |\P_j(E)|^{\frac{1}{k - 1}}, \qquad E\subset\R^k.
\end{align}
Here $|\cdot|$ denotes Lebesgue outer measure and $\P_j:\R^k\to\R^{k-1}$ is the coordinate projection defined by $\P_j(x)=\h{x}_j$, where $\h{x}_j$ is obtained from $x\in\R^k$ by deleting its $j$-th coordinate.

\medskip

Loomis and Whitney proved \eqref{LWc} by a discrete argument \cite{LW49}. The inequality underlies, among other applications, Sobolev inequalities and embeddings \cite{AF03, S02} and the multilinear Kakeya inequality \cite{BCT06, G15}; see also \cite{BN03, BZ88, CGG16, F92} and the references therein. Loomis--Whitney inequalities, as well as the closely related Brascamp--Lieb inequalities, have been approached through rearrangement \cite{BL76, BLL74}, optimal transport \cite{B982, B98, CGS19}, heat-flow monotonicity \cite{BCCT08, CLL04}, and entropy \cite{CC09}.

Recent work has extended Loomis--Whitney and Brascamp--Lieb inequalities to step-two Carnot groups \cite{FP22, Z24, CPT25, HS26, H26}. The noncommutative group law makes the natural coordinate projections nonlinear and obstructs several Euclidean arguments. In particular, the symmetries needed for the heat-flow approach are not necessarily preserved by the relevant flow. Most available proofs therefore use mapping properties of multilinear Radon-type transforms together with induction. An alternative method, introduced for corank-one Carnot groups in \cite{Z24}, uses the duality between Brascamp--Lieb inequalities and entropy subadditivity.

\medskip

The purpose of this paper is to develop the entropy method for the Reiter--Heisenberg groups. When the corank $q$ is larger than one, these groups lie beyond the corank-one setting of \cite{Z24}.

\subsection{Main result}\label{ss13}

For $q,p\in\N^*=\{1,2,3,\ldots\}$, let $\R^{q\times p}$ denote the space of real $q\times p$ matrices. The {\it Reiter--Heisenberg group $\G_{qp}$} is $\R^{q\times p}\times\R^p\times\R^q$ equipped with the group law
\begin{align}\label{glRH}
    (x,y,t) \cdot (x',y',t') = \left(x + x', y + y', t + t' + \frac{1}{2} (xy' - x'y) \right).
\end{align}
It is a step-two Carnot group, as recalled in Subsection \ref{ss21} (with $(x,y)$ denoting variables for the first layer and $t$ for the second layer), and it reduces to a Heisenberg group when $q=1$. Reiter--Heisenberg groups have been studied in several settings; see, for example, \cite{R74, TJ88, M15, MM24}. We identify $\R^{q\times p}$ with $\R^{qp}$ by writing $x=(x_{11},\ldots,x_{q1},\ldots,x_{1p},\ldots,x_{qp})$, where
\[
x = \begin{pmatrix}
  x_{11}& \ldots& x_{1p} \\
   \vdots & \ddots & \vdots \\
   x_{q1}& \ldots& x_{qp}
\end{pmatrix}.
\]
For $1\le l\le p$, set $x(l):=(x_{1l},\ldots,x_{ql})\in\R^q$, and for $1\le k\le q$, let $e_k$ be the $k$-th standard basis vector of $\R^q$. We define
\begin{align}\label{defpiij}
    \pi_{i}(x,y,t) &:= \begin{cases}
    \left(\h{x}_{i}, y, t + \frac{x_{kl} y_l}{2} e_k\right), &\mbox{if } 1 \le i \le qp \mbox{ and } i = (l - 1)q + k, \\
    \left(x,\h{y}_{i - qp}, t - \frac{y_{i - qp}x(i - qp) }{2} \right), &\mbox{if } qp < i \le (q + 1)p
    \end{cases}.
\end{align}
Here $\h{x}_i$ and $\h{y}_j$ are obtained by deleting the indicated coordinates. The maps in \eqref{defpiij} here agree with the projections defined in \eqref{pro1} for general step-two Carnot groups.

\medskip

We write $\|f\|_r$ for the $L^r$ norm of $f$. Our main theorem is the Loomis--Whitney inequality on the Reiter--Heisenberg group $\G_{qp}$ with an explicit constant (not necessarily sharp).

\begin{theorem}\label{t1}
On the Reiter--Heisenberg group $\G_{qp}$, the following Loomis--Whitney inequality holds:
\begin{align}\label{LWco1}
\int_{\G_{qp}} \prod_{j = 1}^{(q + 1)p} f_j(\pi_j(x,y,t))\,dx\,dy\,dt \le \|\RR\|_{\frac{3}{2} \to 3}^{\frac{3q}{qp + p + 2q - 1}} \prod_{j = 1}^{qp} \|f_j\|_{\frac{p(qp + p + 2q - 1)}{1 + p}} \prod_{j = qp + 1}^{(q + 1)p} \|f_j\|_{\frac{p(qp + p + 2q - 1)}{q + p}},
\end{align}
for all nonnegative measurable functions $f_1,\ldots,f_{(q+1)p}$ on $\R^{qp+p+q-1}$. Here $\|\RR\|_{\frac{3}{2}\to3}<+\infty$ denotes the operator norm of the Radon transform $\RR:L^{\frac32}(\R^2)\to L^3(\SS^1\times\R)$.
\end{theorem}
For completeness, the {\it Radon transform} (or {\it X-ray transform}) is defined by
\[
\RR f(\sigma,s):=\int_{\langle z,\sigma\rangle=s}f(z)\,dz,
\qquad (\sigma,s)\in\SS^1\times\R.
\]
Here $\langle \cdot,\cdot\rangle$ is the inner product on Euclidean spaces and $dz$ is one-dimensional Lebesgue measure on the line $\{z\in\R^2:\langle z,\sigma\rangle=s\}$. The required mapping property follows from \cite{OS82}.

\medskip

When $q = p = 1$, i.e. on the first Heisenberg group, the projections in \eqref{defpiij} take the form
\[
\pi_1(x,y,t)=\left(y,t+\frac12xy\right),
\qquad
\pi_2(x,y,t)=\left(x,t-\frac12xy\right),
\]
and the case $q = p = 1$ of Theorem \ref{t1} was proved by F\"assler and Pinamonti, which is the starting point of our proof of Theorem \ref{t1}.

\begin{theorem}[Theorem 2.4 of \cite{FP22}]\label{tH1}
For all nonnegative measurable functions $f_1,f_2$ on $\R^2$, we have
\begin{align}\label{LWH1}
\int_{\R^3} f_1\left(y,t+\frac12xy\right)f_2\left(x,t-\frac12xy\right)\,dx\,dy\,dt
\le \|\RR\|_{\frac32\to3}\|f_1\|_{\frac32}\|f_2\|_{\frac32}.
\end{align}
\end{theorem}
Starting from Theorem \ref{tH1}, in \cite{FP22}, the authors treated higher-dimensional Heisenberg groups by induction on estimates associated with extreme points of the Newton polytope from \cite[Section 3]{S11}, followed by multilinear interpolation. The entropy method of \cite{Z24} provides another route in the corank-one case.

\begin{remark}\label{noniso}
The terms in the second product on the right-hand side of \eqref{LWco1} have a different $L^p$-exponent from the terms in the first product. This anisotropy arises because the left-invariant vector fields associated with the $y$-variables are used repeatedly to generate nontrivial elements of $\g_2$. The same phenomenon occurs for corank-one Carnot groups \cite{Z24}. More precisely, \cite[Theorem 2]{S11} rules out estimates whose associated points lie outside the Newton polytope defined in \cite[Section 3]{S11} (which is related to the bracket-generating structure of the underlying vector fields). The point corresponding to \eqref{LWco1} lies on the boundary of that polytope, where \cite[Theorem 3]{S11} does not apply.
\end{remark}

\subsection{Proof strategy and organization}\label{ss15}

The main tool in the proof is the duality between Brascamp--Lieb inequalities and entropy subadditivity. To be more precise, the proof of Theorem \ref{t1} has two stages. First, we establish Theorem \ref{t1} for $p=1$. Note that the groups $\G_{q1}$ also arise as important examples in \cite{Li21, NZ24}. We combine the first-Heisenberg estimate in Theorem \ref{tH1} with properties of the entropy to carry out this step. 

Second, we identify $\G_{qp}$ with the central sum of $p$ copies of $\G_{q1}$. We prove a general stability theorem showing how a Loomis--Whitney inequality on a step-two Carnot group passes to any finite central sum of copies of that group (which generalizes the result in \cite{Z24}). Applying this result with the inequality from the first stage gives the desired Loomis--Whitney inequality with an explicit constant.

\medskip

The novelty of the argument in this work is the use of the invariance of differential entropy under volume-preserving diffeomorphisms (see Lemma \ref{ldiff} below) in the first stage of the proof. This allows the entropy argument to extend beyond the corank-one setting treated in  \cite{Z24}.

\medskip

In Section \ref{s2} we recall some basic facts on step-two Carnot groups, differential entropy, Brascamp--Lieb inequalities, and the duality between Brascamp--Lieb inequalities and entropy subadditivity. In Section \ref{s3} we prove Theorem \ref{t1}. Finally in Section \ref{s4} we give the resulting geometric projection, Gagliardo--Nirenberg--Sobolev, and isoperimetric inequalities.

\section{Preliminaries}\label{s2}

\subsection{Step-two Carnot groups}\label{ss21}

Recall that a connected, simply connected Lie group $\G$ is a {\it step-two Carnot group} if its Lie algebra $\mathfrak{g}$ admits a stratification
\begin{align}\label{defstr}
\mathfrak{g} = \mathfrak{g}_1 \oplus \mathfrak{g}_2, \quad
[\mathfrak{g}_1, \mathfrak{g}_1] = \mathfrak{g}_2, \quad
[\mathfrak{g}_1, \mathfrak{g}_2] = \{0\},
\end{align}
where $[\cdot,\cdot]$ denotes the Lie bracket on $\mathfrak{g}$. Via the Lie exponential map, we identify $\G$ with $\R^n\times\R^m$, where $n,m\in\N^*$, and write the group law as
\[
(x, t) \cdot (x', t') = \left(x + x', t + t' + \frac{1}{2}\langle \U \, x, x' \rangle \right),
\]
where $\langle\U x,x'\rangle:=(\langle U^{(1)}x,x'\rangle,\ldots,\langle U^{(m)}x,x'\rangle)\in\R^m$; see, for example, \cite[Section~3.2]{BLU07}. Here $\U=\{U^{(1)},\ldots,U^{(m)}\}$ is an $m$-tuple of linearly independent, real, skew-symmetric $n\times n$ matrices, and $\langle\cdot,\cdot\rangle$ is the Euclidean inner product. The numbers $n$ and $m$ are called the {\it rank} and {\it corank}, respectively. We write $g=(x,t)$ for an element of $\G$. The identity is $0=(0,0)$, and the inverse of $g$ is $-g$. If $U^{(j)}=(U^{(j)}_{l,k})_{1\le l,k\le n}$ for $1\le j\le m$, the {\it the canonical left-invariant vector fields} are given by
\begin{align*}
\quad \X_l(g) := \frac{\partial}{\partial x_l} + \frac{1}{2} \sum_{j = 1}^m \left(\sum_{k = 1}^{n} U^{(j)}_{l, k} \, x_k \right) \frac{\partial}{\partial t_j}, \qquad \forall \, 1 \le l \le n.
\end{align*}

\medskip

Following \cite{FP22, Z24}, we define the relevant nonlinear projections on $\G$. For $1\le j\le n$, let $\L_j:=\R e_j$, where $e_j$ is the $j$-th standard basis vector of $\R^n$, and set
\[
\J_j := \{(x,t) \in \G : \, x_j = 0\}, \quad \forall \, 1 \le j \le n.
\]
Fixing $1 \le j \le n$, every $(x,t)\in\G$ has a unique decomposition $(x,t)=(y,s)\cdot\ell e_j$ with $(y,s)\in\J_j$ and $\ell e_j\in\L_j$. We identify $\J_j$ with $\R^{n+m-1}$ by deleting its $j$-th coordinate. The {\it $j$-th projection on $\G$}, $\pi_j:\G\cong\R^{n+m}\to\R^{n+m-1}$, sends $(x,t)$ to the $\J_j$ component of this decomposition. Explicitly,
\begin{align}\label{pro1}
\pi_j(x,t) &= \left(\h{x}_j,t - \frac{1}{2} \sum_{i = 1}^n x_i x_j\langle \U e_i, e_j\rangle \right),  \qquad 1 \le j \le n.
\end{align}
Recall that $\h{x}_j$ is obtained from $x\in\R^n$ by deleting its $j$-th coordinate. For the geometric interpretation and applications of these projections, see \cite[Lemma 1 and Remark 3]{Z24} and \cite{KLZZ25}, respectively. 

\medskip

The {\it dilations} on $\G$ are defined by
\begin{align}\label{defdi}
\delta_r(x,t) = (rx, r^2 t), \qquad \forall \, r > 0, (x,t) \in \G.
\end{align}
Each $\J_j$ inherits a dilation structure from $\G$. After identifying $\J_j$ with $\R^{n+m-1}$, define the {\it $j$-th dilation} $\delta^{(j)}_r$ by requiring that
\begin{align}\label{dlc}
\delta^{(j)}_r \circ \pi_j = \pi_j \circ \delta_r, \qquad \forall \, r > 0, 1 \le j \le n. 
\end{align}
The {\it homogeneous dimension} of $\G$ is $Q:=n+2m$. With $|\cdot|$ denoting Lebesgue measure, we have
\begin{align}\label{homoL}
    |\delta_r (A)| = r^Q |A|, \quad |\delta^{(j)}_r(B)| = r^{Q - 1} |B|, \quad \forall \, r> 0, A \subset \G \cong \R^{n + m}, B \subset \R^{n + m - 1}, 1 \le j \le n.
\end{align}

\medskip

In \cite[Section 8.1]{Li21}, Li constructed a step-two group from two step-two groups of the same corank by an operation other than direct product; this extends the construction of $\H^2$ from two copies of $\H^1$. Further properties were studied in \cite[Appendix C]{LZ21} and \cite[Section 5]{Z23}. We call this operation the {\it central sum}. Let $\G\cong\R^n\times\R^m$ and $\G'\cong\R^{n'}\times\R^m$ have the same corank, and suppose that the group structure of $\G'$ is determined by $\V:=\{V^{(1)},\ldots,V^{(m)}\}$, so that
\[
(y, s) \cdot (y', s') = \left(y + y', s + s' + \frac{1}{2}\langle \V \, y, y' \rangle \right).
\]
The {\it central sum} $\G\#\G'\cong\R^{n+n'}\times\R^m$ is defined by
\[
((x, y), t) \cdot ((x',y'), t') = \left((x + x', y + y'), t + t' + \frac{1}{2}\langle \U \, x, x' \rangle + \frac{1}{2}\langle \V \, y, y' \rangle \right).
\]
The central sum $\G\#\G'$ again has corank $m$. Iterating the construction defines the central sum of any finite number of step-two Carnot groups with the same corank.

\subsection{Differential entropy}\label{ss22}

The notion of entropy originates in mathematical physics \cite{B64} and information theory \cite{S48}. For the facts about differential entropy used below, see \cite[Chapter 8]{CT06}, \cite{I93}, or \cite[Section 2]{Z24}. Following \cite{CC09, Z24}, our sign convention is the negative of the usual information-theoretic convention. On a measure space $(\Omega,\S,\mu)$, let $f$ be a nonnegative measurable function satisfying $\int_\Omega f\,d\mu=1$. We define its {\it (differential) entropy} by
\begin{align}\label{defen}
S(f) := \int_\Omega f(x) \ln{f(x)} d\mu(x) = \int_{\supp f} f(x) \ln{f(x)} d\mu(x).
\end{align}
Here $\supp f:=\{x\in\Omega:f(x)>0\}$, which is the positivity set rather than the usual topological support. We adopt the convention $0\ln0=0$. The integral in \eqref{defen} need not be finite or even well defined. Unless stated otherwise, we assume that $S(f)\in\R$, equivalently,
\[
\int_{\supp f} f(x) |\ln{f(x)}| d\mu(x) < +\infty.
\]
We then say that the entropy is finite. This holds, in particular, when $f$ is bounded and $\supp f$ has finite measure. If a random variable $X$ taking values in $(\Omega,\S,\mu)$ has density $f$, written briefly as $X\sim f$, we also write $S(X):=S(f)$.

\medskip

Suppose that $X$ and $Y$ take values in $(\Omega,\S,\mu)$ and $(\Omega',\S',\mu')$, respectively, and that $(X,Y)\sim f$ on the corresponding product space. Then $Y\sim f_Y$, where the {\it marginal density} is
\begin{align}\label{dmd}
f_Y(y) := \int_\Omega f(x,y) d\mu(x).
\end{align}
For $y \in \Omega'$ such that $0 < f_Y(y) < +\infty$, the {\it conditional density of $X$ given $Y = y$}  is defined by
\begin{align}\label{dcd}
f(x|y) := \frac{f(x,y)}{f_Y(y)}.
\end{align}
By definition, $\int_\Omega f(\cdot|y)\,d\mu=1$. Hence, for almost every $y\in\supp f_Y$, the {\it conditional entropy of $X$ given $Y=y$} is
\begin{align}\label{dce}
S(X|Y = y) := S(f(\cdot|y))
\end{align}
if it exists (that is, $S(f(\cdot|y)) \in \R$). 

\medskip

Let $f$ be a probability density on $(\Omega,\S,\mu)$, $(M,\M,\nu)$ another measure space, and $p:\Omega\to M$ a measurable map. Suppose that $p_\#(f\,d\mu)\ll\nu$, so there is a density $f_{(p)}$ such that
\begin{align}\label{rnp}
p_\# (f d\mu) = f_{(p)} d\nu.
\end{align}
By the definition of pushforward, every bounded measurable function $\phi$ on $M$ satisfies
\begin{align}\label{pfe}
\int_\Omega \phi(p(x)) f(x) d\mu(x) = \int_M \phi(z) f_{(p)}(z) d\nu(z).
\end{align}
Taking $\phi\equiv1$ shows that $\int_Mf_{(p)}\,d\nu=1$. If $S(f_{(p)})\in\R$, we call it the {\it pushforward entropy under $p$}. When $X\sim f$, we have $p(X)\sim f_{(p)}$ and write $S(p(X)):=S(f_{(p)})$.

\medskip

We collect the required facts about differential entropy below. For a proof, we refer to Propositions 3-5 of \cite{Z24}.

\begin{proposition}\label{pcollection}
Assume $(X,Y) \sim f$ on $(\Omega \times \Omega', \S \times \S', \mu \times \mu')$.
\begin{enumerate}[(i)]
\item  If $S(X,Y), S(Y) \in \R$, then for almost every $y \in \supp f_Y$ we have
\begin{align}\label{relS0}
S(X|Y = y) f_Y(y) + f_Y(y) \ln{f_Y(y)} = \int_\Omega f(x,y) \ln{f(x,y)} d\mu(x).
\end{align}
In particular, we have $S(X|Y = y) \in \R$ for almost every $y \in \supp f_Y$ and 
\begin{align}\label{relS}
S(X,Y) = S(Y) + \int_{\supp f_Y} S(X|Y = y) f_Y(y) d\mu'(y).
\end{align}

\item If $S(X,Y), S(X), S(Y) \in \R$, then we have 
\[
S(X,Y) \ge S(X) + S(Y),
\]
where the equality holds if and only if $X$ and $Y$ are independent.

\item Suppose $p$ is a measurable map from $(\Omega, \S, \mu)$ to $(M, \M, \nu)$.  Let the map $\bar{p}$ be a measurable map from $(\Omega \times \Omega', \S \times \S', \mu \times \mu')$ to $(M \times \Omega', \M \times \S', \nu \times \mu')$ defined by $\bar{p}(x,y) = (p(x),y)$. Assume furthermore that $\bar{p}_\# (f d\mu d\mu') \ll \nu \times \mu'$ and
\[
\bar{p}_\# (f d\mu d\mu') = f_{(\bar{p})} d\nu d\mu'.
\]
Then for almost every $y \in \supp f_Y$, we have $p_\# (f(\cdot|y) d\mu) \ll \nu$ and 
\[
p_\# (f(\cdot|y) d\mu) = f_{(\bar{p})}(\cdot|y) d\nu.
\]
In other words, for almost every $y\in\supp f_Y$, we have $f(\cdot|y)_{(p)}=f_{(\bar p)}(\cdot|y)$. Thus $S(p(X)|Y=y)$ is unambiguous: it is both the pushforward entropy under $p$ of $X$ conditioned on $Y=y$ and the conditional entropy of $p(X)$ given $Y=y$.
\end{enumerate}
\end{proposition}

We also need the following invariance of differential entropy under volume-preserving diffeomorphisms.

\begin{lemma}\label{ldiff}
Let $X$ be a random vector on $\R^k$ with density. If $p:\R^k\to\R^k$ is a diffeomorphism with Jacobian determinant $|Jp|\equiv1$, then
\begin{align}\label{reldiff}
    S(X) = S(p(X)).
\end{align}
\end{lemma}

\begin{proof}
Write $X\sim f$. A change of variables gives
\[
\int_{\R^k} \phi(p(x)) f(x) dx = \int_{\R^k} \phi(y) f(p^{-1}(y)) dy
\]
for every bounded measurable function $\phi$. Comparing this identity with \eqref{pfe} gives $f_{(p)}=f\circ p^{-1}$, and therefore
\[
S(p(X)) = S(f_{(p)} ) = \int_{\R^k} f(p^{-1}(y)) \ln{f(p^{-1}(y))} dy = \int_{\R^k} f(x) \ln f(x) dx = S(f) = S(X),
\]
where the third equality uses the same change of variables.
\end{proof}

\subsection{Brascamp--Lieb inequalities}\label{ss12}

The Brascamp--Lieb inequality is a far-reaching generalization of the Euclidean Loomis--Whitney inequality. It was formulated in \cite{BL76} in the study of optimal constants in Young's inequality. In general, it has the form
\begin{align}\label{BL}
\int_{\R^k} \prod_{j  =1}^m f_j^{q_j} (L_j (x)) dx \le \BL(\bL,\q) \prod_{j = 1}^{m} \left( \int_{\R^{k_j}} f_j(t) dt \right)^{q_j}, 
\end{align}
for all nonnegative measurable functions $f_j$ on $\R^{k_j}$, $1\le j\le m$. Here $k,m\in\N^*$, while $q_j\ge0$, $k_j\in\N^*$, and $L_j:\R^k\to\R^{k_j}$ is a linear surjection for each $j$. We write $\bL:=(L_1,\ldots,L_m)$ and $\q:=(q_1,\ldots,q_m)$, and call $(\bL,\q)$ a {\it Brascamp--Lieb datum}. The {\it Brascamp--Lieb constant} $\BL(\bL,\q)$ is the smallest constant for which \eqref{BL} holds and may equal $+\infty$.

\medskip

Lieb's theorem \cite[Theorem 6.2]{L90} shows that it suffices to optimize the Brascamp--Lieb constant over centered Gaussian inputs. The following finiteness criterion is due to \cite[Theorem 1.13 and Proposition 2.8]{BCCT08}; see also \cite[Theorem 6]{B982}.

\begin{theorem}\label{tfBL}
Let $(\bL,\q)$ be a Brascamp--Lieb datum. Then $\BL(\bL,\q)$ is finite if and only if the scaling condition
\begin{align}\label{scaling}
 k = \sum_{j = 1}^m q_j k_j,
\end{align}
and the dimension condition 
\begin{align}\label{dim}
\dim(V) \le \sum_{j = 1}^m q_j \dim(L_j V), \qquad \forall   \ \mbox{subspace $V \subset \R^k$}.
\end{align}
hold. In particular, $\BL(\bL,\q)=1$ if the geometric conditions
\begin{align}\label{geoc}
L_j L_j^* = \id_{k_j}, \quad \forall \, 1 \le j \le m, \quad  \sum_{j = 1}^m q_j L_j^* L_j = \id_k.
\end{align}
Here $\id_k$ denotes the identity matrix on $\R^k$.
\end{theorem}

\medskip

Because the maps $\pi_j$ on step-two groups are generally nonlinear, we use a nonlinear Brascamp--Lieb inequality. The following result was proved by an induction-on-scales argument in \cite[Theorem 1.1]{BBBCF20}; see also \cite{BCW05} for nonlinear Loomis--Whitney inequalities.

\begin{theorem}\label{tnBL}
Let $(\bL,\q)$ be a Brascamp--Lieb datum and suppose that $B_j : \R^k \to \R^{k_j}$ are $C^2$ submersions in a neighborhood of a point $x_0$ and $d B_j(x_0) = L_j$ for $1 \le j \le m$. Then for every $\epsilon > 0$ there exists a neighborhood $U$ of $x_0$ such that 
\begin{align}\label{locBL}
\int_U   \prod_{j  =1}^m f_j^{q_j} (B_j (x)) dx  \le (1 + \epsilon) \BL(\bL,\q) \prod_{j = 1}^{m} \left( \int_{\R^{k_j}} f_j(t) dt \right)^{q_j}
\end{align}
for all nonnegative measurable functions $f_j$ on $\R^{k_j}$, $j=1,\ldots,m$.
\end{theorem}

However, we cannot apply Theorem \ref{tnBL} directly to $\pi_j, 1 \le j \le n$, since it lies in the multilinear Radon-transform regime of \cite{S11, TW03}, where the corresponding linear Brascamp--Lieb constant is not finite. Fortunately, adding the auxiliary projection $\pi_{n+1}(x,t):=x$ yields the following estimate, analogous to \cite[Proposition 1]{Z24}.

\begin{proposition}\label{pWL}
On every step-two Carnot group $\G$, the following inequality holds:
\begin{align}\label{pLWco1}
\int_{\G} \prod_{j = 1}^{n + 1} f_j(\pi_j(x,t))\,dx\,dt \le \prod_{j = 1}^{n + 1} \|f_j\|_{n},
\end{align}
for all nonnegative measurable functions $f_1,\ldots,f_n$ on $\R^{n+m-1}$ and every nonnegative measurable function $f_{n+1}$ on $\R^n$.
\end{proposition}

\begin{proof}
For $1\le j\le n$, we have $d\pi_j(0)=\P_j$, while $d\pi_{n+1}(0)=\pi_{n+1}=: \P_{n+1}$. Direct calculation gives
\[
\P_j\P_j^* = \id_{n + m - 1}, \quad 1 \le j \le n, \quad \P_{n + 1} \P_{n + 1}^* = \id_{n}, \quad   \frac{1}{n} \sum_{j = 1}^{n + 1} \P_j^* \P_j = \id_{n + m}.
\]
Thus the datum is geometric, and Theorem \ref{tfBL} gives Brascamp--Lieb constant $1$. By Theorem \ref{tnBL}, for every $\epsilon>0$ there is a neighborhood $U$ of the origin such that, after replacing the functions in that theorem by $f_j^n$,
\begin{align}\label{nnn}
\int_U \prod_{j=1}^{n+1}f_j(\pi_j(x,t))\,dx\,dt
\le(1+\epsilon)\prod_{j=1}^{n+1}\|f_j\|_n    
\end{align}
for all nonnegative measurable functions $f_1,\ldots,f_n$ on $\R^{n+m-1}$ and every nonnegative measurable function $f_{n+1}$ on $\R^n$. Define $\delta_r^{(n+1)}(x)=rx$ on $\R^n$. Then
\begin{align}\label{dlc2}
\delta_r^{(n + 1)} \circ \pi_{n + 1} = \pi_{n + 1} \circ \delta_r.
\end{align}
For $r>0$, replacing $f_j$ by $f_j\circ\delta_r^{(j)}$ in \eqref{nnn} and a change of variables gives
\[
\int_{\delta_r(U)} \prod_{j=1}^{n+1}f_j(\pi_j(x,t))\,dx\,dt
\le(1+\epsilon)\prod_{j=1}^{n+1}\|f_j\|_n.
\]
The dilation factors cancel by \eqref{dlc} and \eqref{dlc2}. Letting $r\to+\infty$ and then $\epsilon\to0^+$ proves \eqref{pLWco1}.
\end{proof}

\subsection{Duality for Brascamp--Lieb inequalities}\label{ss23}

Our main tool is the duality between Brascamp--Lieb inequalities and entropy subadditivity established in \cite[Theorem 2.1]{CC09}, following the implementation in \cite{Z24}. Finiteness of the relevant entropies must be tracked carefully to avoid undefined expressions such as $-\infty+\infty$. We therefore work with a class of probability densities that is stable under the projections used below.

\begin{theorem}\label{tcc09}
Let $(\Omega,\S,\mu)$ be a measure space. For $1\le j\le m$, let $(M_j,\M_j,\nu_j)$ be a measure space and $p_j:\Omega\to M_j$ a measurable map. Fix $D\in\R$ and $c_j>0$ for $1\le j\le m$. If all nonnegative measurable functions $f_j:M_j\to[0,+\infty)$ satisfy
\begin{align}\label{gbl}
\int_\Omega \prod_{j = 1}^m f_j(p_j(x)) d\mu(x) \le e^D \prod_{j = 1}^m \left( \int_{M_j} f_j^{1/c_j}(t) d\nu_j(t)\right)^{c_j},
\end{align}
then the entropy inequality
\begin{align}\label{sae}
\sum_{j = 1}^m c_j S(f_{(p_j)}) \le S(f) + D
\end{align}
holds for every probability density $f$ in the class
\begin{align}\label{defw}
\W := \{f : \,  S(f) \in \R, \mbox{ and } \forall \, 1 \le j \le m, (p_j)_\#(f d\mu) =  f_{(p_j)} d\nu_j \mbox{ with } S(f_{(p_j)}) \in \R \}.
\end{align}
\end{theorem}

\medskip

For the converse direction on $\R^k$, define
\begin{align}\label{defDk}
 \DP := \left\{f:\, f\ge0,\ f\mbox{ is bounded},\ \supp f\mbox{ is bounded},\ \int f = 1\right\}.
\end{align}
The class $\DP$ behaves well under Euclidean projections and under the nonlinear projections in \eqref{pro1}. The required stability facts and the following converse can be proved exactly as in \cite[Lemmas 2--3 and Theorem 6]{Z24}, so we omit the proofs.
\begin{theorem} \label{cr1}
Let $\G$ be a step-two Carnot group with projections $\{\pi_j\}_{j=1}^n$, and fix $D\in\R$ and $c_j>0$ for $1\le j\le n$. Suppose that every $f\in\DP$ satisfies
\begin{align}\label{sae2}
\sum_{j = 1}^n c_j S(f_{(\pi_j)}) \le S(f) + D.
\end{align}
Then for all nonnegative measurable functions $f_1, \ldots, f_n$ on $\R^{n + m - 1}$ we have
\begin{align}\label{gbl2}
\int_{\G} \prod_{j = 1}^{n} f_j(\pi_j(x,t))\,dx\,dt \le e^D \prod_{j = 1}^{n} \left( \int_{\R^{n + m - 1}} f_j^{1/c_j}(\h{x}_j,t)\,d\h{x}_j\,dt \right)^{c_j}.
\end{align}
\end{theorem}

\begin{remark}\label{finite}
By the stability properties of $\DP$ just cited, all unconditional entropies used below are finite. Proposition \ref{pcollection}(i) then ensures that the conditional entropies under the integral signs are finite almost everywhere on the support.
\end{remark}

\section{Proof of Theorem \ref{t1}}\label{s3}

\subsection{The case \texorpdfstring{$p=1$}{p=1}}\label{ss31}

In the proof we always assume $q \ge 2$, otherwise there is nothing to prove.  Identify $\R^{q\times1}$ with $\R^q$ and write $x=(x_1,\ldots,x_q)$. The projections in \eqref{defpiij} are then
\begin{align}
    \pi_i(x,y,t) = \left(\h{x}_{i},y,t + \frac{x_i y}{2}e_i\right), \quad 1\le i\le q,
    \qquad \pi_{q+1}(x,y,t)=\left(x,t-\frac{y}{2}x\right).
\end{align}
For $1\le i\le q$, define a diffeomorphism $\varphi_i$ of $\R^{2q}\cong\R^{q-1}\times\R\times\R^q$ by
\begin{align}\label{defvp}
    \varphi_i\left(\h{x}_{i}, y, t \right) := \left(\h{x}_{i}, y, t - \sum_{j \ne i} \frac{x_{j} y}{2} e_j\right).
\end{align}
A direct calculation gives $|J\varphi_i|\equiv1$. We also define the diffeomorphism $\h{\varphi}_i$ of $\R^{2q+1}\cong\R^q\times\R\times\R^q$ by
\begin{align}\label{defvp2}
    \h{\varphi}_i(x,y,t) := \left(x,y,t + \sum_{j \ne i} \frac{x_jy}{2}e_j\right),
\end{align}
which also satisfies  $|J\h{\varphi}_i|\equiv1$. Finally, let
\begin{align}\label{defP}
    \Phi_i(x,y,t) := \left(\h{x}_i,\P_i\left(t-\sum_{j\ne i}\frac{x_jy}{2}e_j\right)\right),
\end{align}
where $\P_i(x)=\h{x}_i$ for $x\in\R^q$. Since $\P_i(e_i)=0$, it holds that 
\begin{align}\label{defP2}
\Phi_i(x,y,t) = \left(\h{x}_i,\P_i\left(t-\frac{y}{2}x\right)\right) =  \left(\h{x}_i,\P_i\left(t-\sum_{j\ne i}\frac{x_jy}{2}e_j + \frac{x_iy}{2}e_i\right)\right).   
\end{align}
Set $\CC_0:=\ln\bigl(\|\RR\|_{\frac32\to3}\bigr)$. By Theorem \ref{cr1}, it suffices to prove
\begin{align}\label{main0}
\sum_{j = 1}^q \frac{2}{3q}  S(f_{(\pi_j)}) + \frac{q + 1}{3q} S(f_{(\pi_{q + 1})}) \le S(f) + \CC_0, \qquad \forall \, f \in \DP.    
\end{align}
We henceforth assume $q\ge2$. For ease of notation, let $(X,Y,T)\sim f\in\DP$, where $X=(X_1,\ldots,X_q)$ and $T=(T_1,\ldots,T_q)$. Fix $1\le i\le q$. Applying Lemma \ref{ldiff} to $\varphi_i$ and then using Proposition \ref{pcollection}(i), we obtain
\begin{align}\nonumber
    &S(\pi_i(X,Y,T)) = S(\varphi_i(\pi_i(X,Y,T))) \\
    \label{eee}
    = \, &S(\Phi_i(X,Y,T)) + \int_{\supp f_{\Phi_i(X,Y,T)}} S\left(Y, T_i + \frac{1}{2} X_i Y \Big| \Phi_i(X,Y,T) = \eta\right)  f_{\Phi_i(X,Y,T)}(\eta) d\eta,
\end{align}
and, similarly,
\begin{align}\nonumber
&S(\pi_{q + 1}(X,Y,T))  \\
\label{eee2}
= \, &S(\Phi_i(X,Y,T)) + \int_{\supp f_{\Phi_i(X,Y,T)}} S\left(X_i, T_i - \frac{1}{2} X_i Y \Big| \Phi_i(X,Y,T) = \eta\right)  f_{\Phi_i(X,Y,T)}(\eta) d\eta.
\end{align}
Applying Theorems \ref{tH1} and \ref{tcc09} gives, for almost every $\eta\in\supp f_{\Phi_i(X,Y,T)}$,
\begin{align*}
&S\left(Y, T_i + \frac{1}{2} X_i Y \Big| \Phi_i(X,Y,T) = \eta\right)  +  S\left(X_i, T_i - \frac{1}{2} X_i Y \Big| \Phi_i(X,Y,T) = \eta\right)  \\
\le & \, \frac{3}{2} \left( S\left(X_i, Y, T_i  \Big| \Phi_i(X,Y,T) = \eta\right) + \CC_0 \right).
\end{align*}
Adding \eqref{eee} and \eqref{eee2}, using the preceding conditional inequality, and then applying Proposition \ref{pcollection}(i) and Lemma \ref{ldiff} to $\h{\varphi}_i$ yields
\begin{align}\nonumber
&S(\pi_{i}(X,Y,T)) + S(\pi_{q + 1}(X,Y,T)) \\
\nonumber
\le & \,2 S(\Phi_i(X,Y,T))  + \frac{3}{2} \int_{\supp f_{\Phi_i(X,Y,T)}} S\left(X_i, Y, T_i  \Big| \Phi_i(X,Y,T) = \eta\right) f_{\Phi_i(X,Y,T)}(\eta) d\eta + \frac{3}{2} \CC_0 \\
\label{main1}
= & \frac{1}{2} S(\Phi_i(X,Y,T)) + \frac{3}{2} S(X,Y,T) + \frac{3}{2} \CC_0.
\end{align}
Summing \eqref{main1} over $i=1,\ldots,q$ gives
\begin{align}\label{main2}
\sum_{i = 1}^{q} S(\pi_i(X,Y,T)) + q S(\pi_{q + 1}(X,Y,T)) \le \frac{1}{2} \sum_{i = 1}^q S(\Phi_i(X,Y,T)) + \frac{3q}{2} S(X,Y,T) + \frac{3q}{2} \CC_0.
\end{align}
For $1\le i\le q$, define $\QQ_i(x,t)=(\h{x}_i,\h{t}_i)$. Then it follows from \eqref{defP2} that
\[
\Phi_i=\QQ_i\circ\pi_{q+1}, \qquad 1\le i\le q,
\]
and
\[
\QQ_i\QQ_i^* =  \id_{2q - 2}, \qquad 1 \le i \le q, \qquad  \frac{1}{q - 1}\sum_{i = 1}^q \QQ_i^* \QQ_i = \id_{2q},
\]
Thus Theorems \ref{tfBL} and \ref{tcc09} give
\begin{align}\label{main3}
\sum_{i = 1}^q S(\Phi_i(X,Y,T))
=\sum_{i = 1}^q S\left(\QQ_i(\pi_{q+1}(X,Y,T))\right)
\le(q-1)S(\pi_{q+1}(X,Y,T)).
\end{align}
Substituting \eqref{main3} into \eqref{main2} and multiplying the result by $2/(3q)$ gives \eqref{main0}.

\subsection{The case \texorpdfstring{$p>1$}{p>1}}

Recall that $x(j)=(x_{1j},\ldots,x_{qj})\in\R^q$ for $1\le j\le p$, and write $y=(y_1,\ldots,y_p)$. The map
\begin{align*}
  \G_{qp} &\longrightarrow \underbrace{\G_{q1} \# \cdots \# \G_{q1}}_{p \mbox{ times}} \\
  (x,y,t) &\mapsto (((x(1),y_1), \ldots, (x(p),y_p)),t)
\end{align*}
is a group isomorphism. Consequently, it remains to prove the following central-sum theorem, which also generalizes the procedure in \cite{Z24}.

\begin{theorem}\label{tz24}
Assume that a step-two Carnot group $\G\cong\R^n\times\R^m$ satisfies
\begin{align}\label{pLGG}
\int_{\G} \prod_{j = 1}^{n} f_j(\pi_j(x,t))\,dx\,dt \le C \prod_{j = 1}^{n } \|f_j\|_{1/c_j}
\end{align}
for all nonnegative measurable functions $f_1,\ldots,f_n$ on $\R^{n+m-1}$, where $c_j > 0$ and $C<+\infty$. Then, for every $N\in\N^*$,
\begin{align}\label{pLGGN}
\int_{\G_{\#}^N} \prod_{j = 1}^{nN} f_j(\h{\pi}_j(x,t))\,dx\,dt \le C^{\frac{Q - 1}{n(N - 1) + Q - 1}} \prod_{j = 1}^{nN} \|f_j\|_{1/\h{c}_j}
\end{align}
for all nonnegative measurable functions $f_1,\ldots,f_{nN}$ on $\R^{nN+m-1}$. Here $\G_{\#}^N=\underbrace{\G\#\cdots\#\G}_{N\text{ times}}$, $\{\h{\pi}_j\}_{j=1}^{nN}$ are the corresponding projections on $\G_{\#}^N$, $Q=n+2m$ is the homogeneous dimension of $\G$, and
\[
\h{c}_j=\frac{c_k(Q-1)+N-1}{N\bigl(n(N-1)+Q-1\bigr)}
\]
when $j=(l-1)n+k$, $1\le l\le N$, and $1\le k\le n$.
\end{theorem}

\begin{proof}
The case $N=1$ is immediate, so in the following we assume $N \ge 2$. For $x\in\R^{nN}$, write $x=(x^1,\ldots,x^N)$ with $x^j\in\R^n$, and let $\h{\P}_j$ be the projection that removes the block $x^j$. For ease of notation, let $(X,T)\sim f\in\DP$, where $X=(X^1,\ldots,X^N)$. Proposition \ref{pcollection}(i) gives, for $1\le l\le N$ and $1\le k\le n$,
\begin{align*}
S(\h{\pi}_{(l - 1)n + k}(X,T)) = S(\h{\P}_l(X)) + \int_{\supp f_{\h{\P}_l(X)} }S\left(\pi_k(X^l,T) \Big| \h{\P}_l(X) = \eta\right)  f_{\h{\P}_l(X)}(\eta) d\eta.
\end{align*}
Note that since  $C$ is finite, a scaling argument (using \eqref{homoL}) for \eqref{pLGG} yields $\sum_{k=1}^nc_k=Q/(Q-1)$. Multiplying by $c_k$ and summing over $k$ therefore gives
\begin{align*}
    &\sum_{k = 1}^n c_k S(\h{\pi}_{(l - 1)n + k}(X,T)) \\
    =  \, &\frac{Q}{Q - 1} S(\h{\P}_l(X)) + \int_{\supp f_{\h{\P}_l(X)} }\sum_{k = 1}^n c_kS\left(\pi_k(X^l,T) \Big| \h{\P}_l(X) = \eta\right)  f_{\h{\P}_l(X)}(\eta) d\eta,
\end{align*}
Moreover, \eqref{pLGG} and Theorem \ref{tcc09} imply that
\begin{align*}
\sum_{k = 1}^n c_k S\left(\pi_k(X^l,T) \Big| \h{\P}_l(X) 
=  \eta\right) \le 
S\left(X^l,T \Big| \h{\P}_l(X) = \eta\right) + D,
\end{align*}
where $D=\ln C$. Substituting this estimate into the preceding identity and applying Proposition \ref{pcollection}(i) again yields
\begin{align*}
  \sum_{k = 1}^n c_k S(\h{\pi}_{(l - 1)n + k}(X,T)) \le \frac{1}{Q - 1} S(\h{\P}_l(X)) + S(X,T) + D.
\end{align*}
Summing over $l$ gives
\begin{align}\label{main01}
    \sum_{j = 1}^{nN} d_j S(\h{\pi}_{j}(X,T)) \le \frac{1}{Q - 1} \sum_{l = 1}^N S(\h{\P}_l(X)) + NS(X,T) + ND,
\end{align}
where $d_j = c_k$ for $j = (l - 1)n + k$ with $1 \le l \le N$ and $1 \le k \le n$. 
Since
\[
\h{\P}_l\h{\P}_l^* =  \id_{n(N- 1)}, \qquad 1 \le l \le N, \qquad  \frac{1}{N - 1}\sum_{l = 1}^N \h{\P}_l^* \h{\P}_l = \id_{nN},
\]
Theorems \ref{tfBL} and \ref{tcc09} imply
\begin{align*}
\sum_{l = 1}^N S(\h{\P}_l(X))    \le   ( N - 1) S(X).
\end{align*}
Substituting this into \eqref{main01} yields
\begin{align}\label{main02}
\sum_{j = 1}^{nN} d_j S(\h{\pi}_{j}(X,T)) \le \frac{N - 1}{Q - 1} S(X) + NS(X,T) + ND.    
\end{align}
Finally, Proposition \ref{pWL} and Theorem \ref{tcc09}, applied on $\G_\#^N$ with the auxiliary projection $(x,t)\mapsto x$, give
\[
\sum_{j = 1}^{nN} S(\h{\pi}_{j}(X,T)) + S(X) \le nN S(X,T).
\]
Set $a:=(N-1)/(Q-1)$. Combining the estimate above with \eqref{main02} gives
\[
\sum_{j=1}^{nN}(d_j+a)S(\h{\pi}_j(X,T))
\le N(1+na)S(X,T)+ND.
\]
Dividing by $N(1+na)=N\bigl(n(N-1)+Q-1\bigr)/(Q-1)$ gives
\begin{align*}
  \sum_{j = 1}^{nN} \h{c}_j S(\h{\pi}_{j}(X,T)) \le S(X,T) + \frac{D(Q - 1)}{n(N - 1) + Q - 1}.   
\end{align*}
Now inequality \eqref{pLGGN} follows from Theorem \ref{cr1} .
\end{proof}

\begin{proof}[Proof of Theorem \ref{t1} for the case $p > 1$]
We now apply Theorem \ref{tz24} with
\[
\G=\G_{q1},\qquad n=q+1,\qquad m=q,\qquad Q=3q+1,
\qquad N=p,
\]
and
\[
C=\|\RR\|_{\frac32\to3},\qquad
c_k=\frac{2}{3q},\quad 1\le k\le q,\qquad
c_{q+1}=\frac{q+1}{3q}.
\]
Put
\[
A:=n(N-1)+Q-1=(q+1)(p-1)+3q=qp+p+2q-1.
\]
Theorem \ref{tz24} gives
\[
\h{c}_{(l-1)(q+1)+k}=\frac{p+1}{pA},
\qquad 1\le l\le p,\ 1\le k\le q,
\]
and
\[
\h{c}_{l(q+1)}=\frac{q+p}{pA},
\qquad 1\le l\le p,
\]
with constant $\|\RR\|_{\frac32\to3}^{3q/A}$. Under the preceding identification of $\G_{qp}$ with the central sum, the projections $\h{\pi}_{(l-1)(q+1)+k}$ for $1\le k\le q$ correspond to $\pi_{(l-1)q+k}$, while $\h{\pi}_{l(q+1)}$ corresponds to $\pi_{qp+l}$. Consequently, the first $qp$ norm exponents are $pA/(p+1)$ and the last $p$ norm exponents are $pA/(q+p)$. This is precisely \eqref{LWco1} and completes the proof of Theorem \ref{t1}.    
\end{proof}

\section{Applications}\label{s4}


We record consequences of the Loomis--Whitney inequality for step-two Carnot groups. The deductions are standard, so we omit their proofs. For proofs and historical background, see \cite{CDPT07, FP22, S02, Z24} and the references therein. In the set-valued statement below, $|\cdot|$ denotes Lebesgue outer measure.

\begin{proposition}
Assume that a step-two Carnot group $\G$ satisfies a Loomis--Whitney inequality of the form
\begin{align*}
\int_{\G} \prod_{j = 1}^{n} f_j(\pi_j(x,t))\,dx\,dt \le C \prod_{j = 1}^{n } \|f_j\|_{1/c_j}
\end{align*}
for all nonnegative measurable functions $f_1,\ldots,f_n$ on $\R^{n+m-1}$, where $c_j > 0$ and $C<+\infty$. Then
\[
|E| \le C \prod_{j = 1}^n |\pi_j(E)|^{c_j}
\]
for every set $E\subset\G$.
\end{proposition}

We next recall the horizontal variation and perimeter. Let
\[
\F(\G):=\{\varphi\in C^1_0(\G,\R^n):|\varphi|\le1\},
\]
where the bound is pointwise. For $f\in L^1(\G)$, define its {\it variation} by
\[
\Var_{\G}(f):=\sup_{\varphi\in\F(\G)}\int_{\G}f(x,t)\sum_{j=1}^n\X_j\varphi_j(x,t)\,dx\,dt.
\]
Let $\BV(\G)$ be the space of functions $f\in L^1(\G)$ with finite variation. It is a Banach space with norm
\[
\|f\|_{\BV( \G)} := \|f\|_1 + \Var_{ \G} (f).
\]
For a measurable set $E$, define its {\it perimeter} by
\[
\PP_{ \G} (E):= \Var_{ \G} (\chi_E).
\] 
See \cite{CDG94, FSS96} for BV functions associated with systems of vector fields and \cite{CDPT07} for the first Heisenberg group $\H^1$.

\begin{proposition}
Assume that a step-two Carnot group $\G$ satisfies a Loomis--Whitney inequality of the form
\begin{align*}
\int_{\G} \prod_{j = 1}^{n} f_j(\pi_j(x,t))\,dx\,dt \le C \prod_{j = 1}^{n } \|f_j\|_{1/c_j}
\end{align*}
for all nonnegative measurable functions $f_1,\ldots,f_n$ on $\R^{n+m-1}$, where $c_j > 0$ and $C<+\infty$. Then there is a constant $C'>0$ such that
\[
\|f\|_{\frac{Q}{Q - 1}} \le C' \Var_{ \G}(f), \qquad \forall \, f \in \BV( \G).
\]
In particular,
\[
|E|^{\frac{Q - 1}{Q}} \le C' \PP_{ \G}(E), \qquad \forall \, E \mbox{ with finite perimeter}.
\]
\end{proposition}

\paragraph*{Acknowledgements.}
 
S.-C. Mao is partially supported by the China Postdoctoral Science Foundation (Grant No. 2026M793367). Y. Zhang has received funding from the European Research Council (ERC) under the European Union’s Horizon 2020 research and innovation programme (grant agreement GEOSUB, No. 945655).

\paragraph*{Competing Interests.}

The authors have no relevant financial or non-financial interests to disclose.

\paragraph*{Declaration of AI Use.}

During the preparation of this manuscript, the authors used OpenAI GPT-5.6 for English language editing and stylistic improvements, including grammar, wording, and readability. All mathematical content was developed and verified by the authors, who take full responsibility for the final manuscript.

\bibliographystyle{abbrv}
\bibliography{LWbib}

\mbox{}\\
Sheng-Chen Mao\\
School of Mathematics and Statistics \\
Lanzhou University \\
No. 222 Tianshui South Road \\
Lanzhou 730000, P.R. China \\
Email: maoshengchen@lzu.edu.cn \quad or \quad maosci@163.com  

\mbox{}\\
Ye Zhang (\textit{corresponding author})\\
SISSA  \\
via Bonomea 265 \\
34136 Trieste, Italy \\
Email: yezhang@sissa.it \quad or \quad zhangye0217@gmail.com 

\end{document}